\documentclass[12pt, reqno]{amsart}

\usepackage{amsfonts,amssymb,latexsym,amsmath, amsxtra}
\usepackage{verbatim}

\usepackage{amsthm,amssymb,amstext,amscd,amsfonts,amsbsy,amsxtra,latexsym}
\usepackage{float}
\usepackage[latin1]{inputenc}
\usepackage[english]{babel}

\newcommand{\field}[1]{\mathbb{#1}}
\newcommand{\N}{\field{N}}
\newcommand{\Z}{\field{Z}}
\newcommand{\R}{\field{R}}

\newcommand{\bea}{\begin{eqnarray}}
\newcommand{\eea}{\end{eqnarray}}
\newcommand{\be}{\begin {equation}}
\newcommand{\ee}{\end{equation}}

\renewcommand{\(}{\left(}
\renewcommand{\)}{\right)}

\numberwithin{equation}{section}
\newtheorem{theorem}{\textbf{Theorem}}
\numberwithin{theorem}{section}
\newtheorem{lemma}[theorem]{\textbf{Lemma}}
\newtheorem{proposition}[theorem]{\textbf{Proposition}}
\newtheorem{corollary}[theorem]{\textbf{Corollary}}
\newtheorem{example}{\textbf{Example}}
\newtheorem{remark}{\textbf{Remark}}

\renewenvironment{proof}[1][Proof]{\begin{trivlist}
\item[\hskip \labelsep {\bfseries #1:}]}{\qed\end{trivlist}}

\begin{document}
\title[Bloch-Okounkov $n$-point functions]{On the Fourier expansion of Bloch-Okounkov $n$-point function}

\author{Kathrin Bringmann} 
\address{Mathematical Institute\\University of
Cologne\\ Weyertal 86-90 \\ 50931 Cologne \\Germany}
\email{kbringma@math.uni-koeln.de}
\author{Antun Milas}
\address{Department of Mathematics and Statistics \\ SUNY-Albany \\ Albany, NY 12222, USA}
\email{amilas@albany.edu}
\date{\today}
\thanks{K.B. was supported by the Alfried Krupp Prize for Young University Teachers of the Krupp foundation and the research leading to these results has received funding from the European Research Council under the European Union's Seventh Framework Programme (FP/2007-2013) / ERC Grant agreement n. 335220 - AQSER}
\thanks{A.M. was partially supported by a Simons Foundation grant}

\begin{abstract} 
In this paper, we study algebraic and analytic properties of Fourier coefficients, expressed as $q$-series, of the so-called Bloch-Okounkov $n$-point function.
We prove several results about these series and explain how they relate to Rogers' false theta function.
Then we obtain their full asymptotics, as $\tau \rightarrow 0$, and use this result to derive asymptotic properties of the coefficients in the $q$-expansion. 
At the end, we also introduce and discuss higher rank generalization of Bloch-Okounkov's functions.

\end{abstract}

\maketitle

\section{Introduction and statement of results}
In \cite{BO}, S. Bloch and A. Okounkov introduced certain formal series $F\(t_1,..,t_n\)$ (defined in Section 2), depending also on $q:=e^{2\pi i\tau}$ ($\tau \in \mathbb{H}$),  in connection with the representation theory of the Lie algebra of differential operators on the circle and quasi-modular forms. 
Somewhat unexpectedly, they showed that $F\(t_1,...,t_n\)$ $(n\in \N)$, which they called the \emph{$n$-point function}, has a meromorphic extension with possible poles at divisors $q^m t_{j_1} \cdots t_{j_k}=1$, where $\{j_1,...,j_k \} \subset \{1,...,n\}$ and $m \in \mathbb{N}$. Among other things, Bloch and Okounkov  proved the following beautiful result.
\begin{theorem} \label{bo-main}
We have
$$F\(t_1,...,t_n\)=\sum_{\sigma \in S_n} \frac{{\rm det}\left(\frac{\Theta^{\(k-j+1\)}\(t_{\sigma\(1\)},...,t_{\sigma\(n-k\)}\)}{\(k-j+1\)!}\right)_{j, k=1}^n}{\Theta\(t_{\sigma\(1\)}\) \Theta\(t_{\sigma\(1\)}t_{\sigma\(2\)}\) \cdots \Theta\(t_{\sigma\(1\)} t_{\sigma\(2\)} \cdots t_{\sigma\(n\)}\)},$$
where $S_n$ is the symmetric group on $n$ letters,
$$\Theta\(t\):=\sum_{\ell\in \mathbb{Z}} \(-1\)^\ell q^{\frac12 \left(\ell+\frac{1}{2}\right)^2} t^{\ell+\frac{1}{2}},$$
and where $\Theta^{(k)}(t):=(t \frac{d}{dt})^k \Theta(t)$. In particular 
$$F(t)=\frac{1}{\Theta(t)}.$$
\end{theorem}
In the second author's PhD thesis (Part 2; published as \cite{M}, see also \cite{M3}), this $n$-point function was  reformulated in the language of vertex algebras and  generalized for an arbitrary insertion of vectors.  
This generalization is based on a simple observation: the $n$-point function above is 
obtained from a $2n$-point correlation function on the torus in conformal field theory (or vertex algebra theory), {\em after} we integrate out
$n$ of the variables. More precisely, it was shown in \cite{M}  that $F(t_1,...,t_n)$ is the constant term of the following elegant quotient 
\begin{equation} \label{modular}
{\frac{\displaystyle{\prod_{1 \leq j < k \leq n} \Theta\left(\frac{t_k x_k}{t_j x_j}\right) \Theta\left(\frac{x_k}{x_j}\right)}}{\displaystyle{{\prod_{j=1}^n \Theta(t_j)} \prod_{1 \leq j < k \leq n} \Theta\left(\frac{t_k x_k}{x_j}\right) \Theta\left(\frac{x_k}{t_j x_j}\right)}}},
\end{equation} 
taken with respect to the $x$-variables

Because we are only considering a ``slice'' of the correlation function,  Bloch-Okounkov functions cannot be elliptic with respect to  $u_j \mapsto u_j+\tau$ if we write $t_j:=e^{2 \pi i u_j}$. Instead, a very complicated $q$-difference 
equation holds, which, in a way, reflects the complexity of the determinant part in Theorem \ref{bo-main}.
The relevant vertex
(super)algebra for $F(t_1,...,t_n)$ is the infinite-wedge fermionic space \cite{Kac}.
There are other interpretations of $n$-point functions closely related to random partitions which we do not pursue here (instead, see \cite{EO,O}).

In this paper, we investigate algebraic and analytic properties of the Fourier coefficients
\begin{equation} \label{coeff}
 {\rm Coeff}_{t^{r_1+\frac12}_1 \cdots t^{r_n+\frac12}_n} F(t_1,...,t_n),
 \end{equation}
where the $r_j$ are integers. 
There are at least two reasons for studying these objects. Firstly, beyond $n$-point correlation functions on the torus, much 
less is known about traces of products of modes of vertex operators which we study here. 
Another reason is recent increased interest in meromorphic Jacobi forms of negative index, including their Fourier coefficients (see \cite{BCR} and references therein). For example, the negative index Jacobi forms $\frac{1}{\Theta(t)^\ell}$ ($\ell \in \mathbb{N}$) studied in \cite{BCR} also appear in this paper as (generalized) higher-level Bloch-Okounkov $1$-point  functions.  The $n$-point functions $F(t_1,...,t_n)$ ($n \geq 2$) are more subtle to describe in terms of Jacobi forms.  From Theorem \ref{bo-main}, we can at least infer that $$F(t_1,...,t_n)=\frac{1}{\Theta(t_1 \cdots t_n)}  \widetilde{F}(t_1,...,t_n),$$ where $\widetilde{F}(t_1,...,t_n)$ 
 is a linear combination of functions of  non-negative ``weight'' $\geq 0$, where the ``weight''  of $\frac{\Theta^{(k)}(\cdot )}{\Theta(\cdot )}$ is set to be $k$ and where 
 empty slots stand for products of $t_j$'s.  Alternatively, we can also view (\ref{coeff}) as a term in the Fourier expansion of (\ref{modular}), which is clearly 
 a quotient of products of Jacobi forms. 
 
 Except for a few examples (implicitly) studied in \cite{M},  
 we know very little about Fourier coefficients of generalized Bloch-Okounkov functions for $\ell \geq 2$ and $n \geq 2$.

In this paper, we establish two main results. We first give a fairly useful description of these coefficients in terms of Rogers false theta function (cf. \cite{R})
$$\sum_{\ell\geq 0} (-1)^\ell q^{\frac{\ell(\ell+1)}{2}}.$$
\begin{theorem} \label{intro-main-formula}
Let
$ r_1 > r_2 > \cdots > r_m \geq 0$ and  $ s_{m+1} >  \cdots > s_n \geq 0$ be sequences of non-negative integers such that  $-r_j-1/2 \neq s_k+1/2$ for all $1 \leq j \leq m$ and $m+1 \leq k \leq n$. Let $G_{{\bf r},{\bf s}}$ $({\bf r}:= (r_1, \dots, r_m), {\bf s}:= (s_{m+1}, \dots s_n))$ be the Fourier coefficient as defined in Section 2.5.
Then the series
$$\prod_{\substack{1 \leq j \leq m \\ m+1 \leq k \leq n}} \left(1-q^{r_j+s_k+1}\right)  \prod_{1 \leq k < \ell \leq m} \left(1-q^{r_k -r_\ell}\right)  \prod_{m+1 \leq k < \ell \leq n} \left(1-q^{s_k -s_\ell}\right)  G_{{\bf r},{\bf s}}(q),$$
can be written as $$P(q)\sum_{\ell \geq 0}(-1)^{\ell}q^{\frac{\ell(\ell+1)}{2}}+ Q(q)$$ for unique Laurent polynomials $P$ and $Q$.
\end{theorem}
Next we find the explicit asymptotic expansion of (\ref{coeff}) in terms of Euler numbers; see Theorem \ref{main}. In particular, this gives
\begin{corollary}
With the ${\bf r}$ and ${\bf s}$ as in Theorem \ref{intro-main-formula}, we have, as $y \rightarrow 0$,
\[
G_{{\bf r},{\bf s}}\left(e^{-2\pi y}\right)\sim\frac1{2^n}\left(1-\pi\left(r-s+m+\frac{n(n-3)}4 \right)y+ O\left(y^2\right)\right).
\]
\end{corollary}

This is then used to obtain the asymptotic behavior of the modes in the $q$-expansion of  (\ref{coeff}) (see Theorem \ref{asym-coeff} and Corollary \ref{asym-coeff-2}). 
We also briefly discuss more general ``higher level'' Bloch-Okounkov $n$-point functions $F^{(\ell)}(t_1,...,t_n)$ $(\ell \in \mathbb{N})$, introduced in the last section.

\section{The Fourier expansion of $F(t_1,...,t_n)$}

In this section we are interested in explicit formulas for the coefficients in the Fourier expansion of $F(t_1,...,t_n)$, that is, we want to 
extract some information about coefficients $t^{r_1+1/2} \cdots t^{r_n+1/2}$ ($r_j\in\Z$). We shall consider several cases 
depending on the nature of the $r_j$'s.

Let us first recall, following notation in \cite{M}, the  definition of $F(t_1,...,t_n)$. 
Denote by $\mathcal{F}$ the infinite-wedge vertex superalgebra  associated to a pair of fermions $\psi(k+1/2), \psi^*(k+1/2)$, $k \in \mathbb{Z}$. 
This space, which also has a 
vertex operator superalgebra structure \cite{Kac}, admits a decomposition into  
"charge" subspaces: $\mathcal{F}=\oplus_{m \in \mathbb{Z}} \mathcal{F}_{m}$,
such that 
 $\mathcal{F}_m$ is spanned by monomials of the form
\begin{multline*}\psi\left(-n_\ell-\frac12\right) \cdots \psi \left(-n_1-\frac12 \right)\psi^*\left(-m_j-\frac12\right) \cdots \psi^*\left(-m_1-\frac12\right){\bf 1} \\ ; \ \ n_\ell > \cdots > n_1, m_j > \cdots > m_1,\end{multline*}
where $m=\ell-j \in \mathbb{Z}$ and ${\bf 1}$ is the vacuum vector. 
We also use the usual convention where the fermonic modes satisfy anti-bracket relations $(\ell, k\in\N)$:
\begin{equation*} \label{bracket}
\left[\psi^*\left(\ell+\frac{1}{2}\right),\psi\left(k+\frac{1}{2}\right)\right]_+=\delta_{\ell+k+1,0},
\end{equation*}
while other anti-commutators are zero. We further let 
\[
A(t):=\sum_{\ell \in \mathbb{Z}} \psi^*\left(\ell+\frac12\right)\psi\left(-\ell-\frac12\right) t^{-\ell-\frac12},
\]
and finally, we define Bloch-Okounkov's $n$-point function
\be \label{bo-def}
F(t_1,...,t_n):=\eta(\tau) \times {\rm tr}_{\mathcal{F}_0} A(t_1) \cdots A(t_n)q^{L(0)-\frac1{24}},
\ee
where $\eta (\tau) := q^{\frac{1}{24}} \prod_{m\geq 1} (1-q^m)$ is \emph{Dedekind's $\eta$-function}.
We note that this trace was computed only on the zero-charge subspace of $\mathcal{F}$; traces on $\mathcal{F}_m$ can be easily computed 
from $F(t_1,...,t_n)$ by a simple shift of parameters, so we omit their considerations here. Observe here that (\ref{bo-def}) has very different properties compared to   
correlation functions studied in vertex algebra theory. For example, in the $t$-expansion of  $A(t)$ all the coefficients have degree zero (if viewed as operators
acting on the graded space $\mathcal{F}_0$).

\subsection{Positive powers: no repetition}
Throughout, we abbreviate for $\text{\bf r} :=(r_1, \dots, r_n)$
$$
F_{\text{\bf r} }(q):= {\rm Coeff}_{t^{r_1+\frac12}_1 \cdots t^{r_n+\frac12}_n} F(t_1,...,t_n),
$$
where the $r_j$'s are non-negative integers.
\begin{theorem}\label{2.1} For $ r_1 > r_2 > \cdots > r_n \geq 0$, we have
$$F_{\text{\bf r} }(q) = {\rm CT}_{\zeta} \left( \prod_{j=1}^n \frac{\zeta^{-1} q^{r_j+\frac{1}{2}} }{\left(1+\zeta^{-1} q^{r_j+\frac{1}{2}}\right)} \sum_{\ell \in \mathbb{Z}} \zeta^\ell q^{\frac{\ell^2}{2}}  \right),$$
where ${\rm CT}$ denotes the constant term with respect to the formal variable $\zeta$ and where throughout $1>|\zeta|>|q^{\frac12}|$.
\end{theorem}
\begin{proof} 
We have two linear operators  
$$\psi^*\left(\ell+\frac{1}{2}\right):  \mathcal{F}_\ell \rightarrow \mathcal{F}_{\ell-1} \ \ {\rm and} \ \  \psi\left(\ell+\frac{1}{2}\right): \mathcal{F}_{\ell-1} \rightarrow \mathcal{F}_{\ell}.$$
Because of the condition on the $r_j$'s,  observe that 
$\psi^*\left(-r_j-\frac12\right)\psi\left(r_j+\frac12\right)$ and \\  $\psi^*\left(-r_k-\frac12\right)\psi\left(r_k+\frac12\right)$
commute with each other for all $j$ and $k$.
Then 
$$F_{\text{\bf r} }(q)=\eta(\tau) \times {\rm tr}_{\mathcal{F}_0} \prod_{j=1}^n \psi^*\left(-r_j-\frac12\right)\psi\left(r_j+\frac12\right) q^{L(0)-\frac1{24}}.$$
This trace can now be computed on the aforementioned basis. Observe that the only monomials that contribute to the trace must contain the submonomial 
$$\psi^*\left(-r_n-\frac12\right) \cdots \psi^*\left(-r_1-\frac12\right),$$
with no other conditions  except that the total charge is zero. Notice that the weight of this submonomial is $\frac{n}{2}+\sum_{j=1}^n r_j$ and its charge is $-n$.
Thus the ${\bf r}$-the coefficient of the $n$-point function equals 
$$F_{\text{\bf r} }(q)=(q;q)_\infty \times  q^{\frac{n}{2}+\sum_{j=1}^n r_j}\ {\rm CT}_\zeta \left( \zeta^{-n} \prod_{\substack{\ell \geq 0 \\ \ell\notin \{r_1,...,r_n \}}} \(1+\zeta^{-1} q^{\ell+\frac 12}\)
\prod_{\ell \geq 0}\(1+\zeta q^{\ell+\frac 12}\) \right).$$
An application of the Jacobi triple product identity 
$$(q;q)_\infty \left(-q^{\frac12}\zeta;q\right)_\infty \left(-q^{\frac12}\zeta^{-1};q\right)_\infty =\sum_{\ell \in \mathbb{Z}}  \zeta^\ell q^{\frac{\ell^2}2},$$ where $(a; q)_\ell := \prod_{j=0}^{\ell-1} (1-aq^j)$ for $\ell\in \N_0 \cup \{\infty\}$, now gives
$$F_{\text{\bf r} }(q) ={\rm CT}_{\zeta} \left( \prod_{j=1}^n \frac{\zeta^{-1} q^{r_j+\frac{1}{2}} }{\left(1+\zeta^{-1} q^{r_j+\frac{1}{2}}\right)} \sum_{\ell \in \mathbb{Z}} \zeta^\ell q^{\frac{\ell^2}{2}}  \right),$$
as desired.
\end{proof}

\begin{remark} \label{rem-1} {\em  
Observe that Bloch-Okounkov's definition of $F_{\text{\bf r} }(q)$ in (\ref{bo-def}) involves multiplication of a graded trace on $\mathcal{F}_0$ with the Dedekind $\eta$-function. As such it involves both 
positive and negative coefficients. Thus it make sense to consider the pure trace and somewhat better behaved function 
$$\frac{F_{\text{\bf r} }(q)}{(q;q)_\infty},$$
which is essentially the generating function of all pairs $(\pi ,\pi')$
where $\pi$ and $\pi'$ are partitions into distinct odd parts such that $\pi'$ has no parts of size $2 r_j+1$ and $\ell(\pi)-\ell(\pi')=n$.
We shall return to this quotient later on when we consider asymptotic properties of $F_{\text{\bf r}}(q)$.}
\end{remark}

We next write $F_{\text{\bf r} }$ as an explicit $q$-series.
\begin{corollary}
Assuming the notation above we have
\begin{eqnarray*}
F_{\text{\bf r} }(q) =\sum_{\substack{m_j \geq 0 \\ 1\leq j\leq n}}  (-1)^{\sum_{k=1}^n m_k} q^{\frac12\left(n+\sum_{k=1}^n m_k\right)^2+ \sum_{k=1}^n \left(m_k+1\right)
\left(r_k+\frac{1}{2}\right)}.
\end{eqnarray*}
\end{corollary}
\begin{proof}
The corollary follows directly from Theorem \ref{2.1} by expanding
$$\frac{\zeta^{-1} q^{r+\frac{1}{2}} }{\left(1+\zeta^{-1} q^{r+\frac{1}{2}}\right)}=\sum_{\ell\geq 0} (-1)^{\ell} \zeta^{-\ell-1} q^{(\ell+1)\left(r+\frac{1}{2}\right)} .$$ 
\end{proof}

We next consider some examples. 
\begin{example}\label{ex1} $(n=1)$ We have for $r\in\N_0$ 
$$F_r(q)=\sum_{\ell\geq 0} (-1)^\ell q^{\frac{(\ell+1)^2}{2}+(\ell+1)\left(r+\frac{1}{2}\right)}.$$ 
Observe that this is, modulo the $\eta$-factor, the character of a module for the Zamolodchikov singlet vertex algebra $\mathcal{W}(2,3)$. 
In the notation of \cite{BM, CM}, this irreducible module is denoted by $M_{r+2,1}$. 
\end{example}

\begin{example} ($n=2$) An easy computation gives with $0\leq r_1<r_2$ 
\begin{align*}
F_{r_1,r_2}(q)
&=\sum\limits_{{\ell\geq 1\atop{0\leq \ell_1<\ell}}}(-1)^{\ell+1}q^{\frac{(\ell+1)^2}{2}+\left(m_1+1\right)\left(r_1+\frac12\right)+\left(\ell-m_1\right)\left(r_2+\frac12\right)}\\
&=\frac{q^{r_1+\frac{1}{2}}}{1-q^{r_1-r_2}}\sum_{\ell>1}(-1)^{\ell+1} q^{\frac{(\ell+1)^2}{2}+\ell\left(r_2+\frac12\right)}\left(1-q^{\ell(r_1-r_2)}\right).
\end{align*}
\end{example}
More generally, we may present the functions of interest in terms of false theta functions.
\begin{theorem} \label{2.2} For $|q|<1$ and $n \geq 2$,  let  $r_1 > r_2 > \cdots > r_n \geq 0$. 
Then the series 
$$\prod_{1 \leq j < k \leq n} (1-q^{r_j -r_k}) F_{\text{\bf r} }(q)$$
is a finite sum of series of the form $\sum_{\ell\geq 0} (-1)^\ell q^{\frac{\ell^2}{2}+b\ell+c}$. More precisely, for every $n \geq 1$, it has the shape $$P(q)\sum_{\ell \geq 0}(-1)^{\ell}q^{\frac{\ell(\ell+1)}{2}}+ Q(q), $$with $P$ and $Q$ Laurent polynomials. 
\end{theorem}
\begin{proof}  We use induction on $n$. 
The second statement is clearly true for $n=1$ since from Example \ref{ex1}
\[
F_{\text{\bf r}}(q) =\sum_{\ell\geq 0} (-1)^\ell q^{\frac{(\ell+1)^2}{2}+(\ell+1)\left(r+\frac{1}{2}\right)}=(-1)^{r+1}q^{-\frac{1}{2}r(r+1)} \sum_{\ell \geq 1+r} (-1)^{\ell} q^{\frac{\ell(\ell+1)}{2}}.
\]
The first statement of the theorem also holds for $n=2$.
Instead of summing over the orthant $\mathbb{N}_{0}^n$, we introduce variables
$\lambda_j:=m_1+\cdots+m_j$, where $\lambda_n \in \mathbb{N}_{0}$ and $(\lambda_1,...,\lambda_{n-1})$ 
parametrize the $(n-1)$-simplex. Clearly, $m_j=\lambda_j-\lambda_{j-1}$. Rewriting the sum in terms of new variables
results in 
\[
F_{\text{\bf r}}(q)=\sum_{{\lambda_{n}\geq 0\atop{0\leq \lambda_j\leq \lambda_{j+1}}}\atop{0\leq j\leq n-1}}
(-1)^{\lambda_n}q^{\frac{\left(\lambda_n+n\right)^2}{2}+\left(\lambda_1+1\right)\left(r_1+\frac12\right)+
\sum_{r=2}^n 
\left(\lambda_k-\lambda_{k-1}+1\right)\left(r_k+\frac12\right)}.
\]
After we sum over $\lambda_1$, we get that this equals
\[
q^{r_1+\frac{1}{2}}\sum_{{\lambda_n\geq 0\atop{0\leq \lambda_j\leq \lambda_{j+1}}}\atop{2\leq j\leq n-1}} 
 \frac{\left(1-q^{(\lambda_2+1)(r_1-r_2)}\right)}{1-q^{r_1-r_2}}(-1)^{\lambda_n}  q^{\frac12\left(\lambda_n+n\right)^2+(\lambda_2+1)\left(r_2+\frac12\right)+\sum_{r=3}^n 
\left(\lambda_k-\lambda_{k-1}+1\right)\left(r_k+\frac12\right)}.
\]
Now define
\[
F_{a,n}\left(r_1,\dots,r_n\right):= \sum_{{\lambda_n\geq 0
\atop{0\leq \lambda_j\leq \lambda_{j+1}}}\atop{0\leq j\leq n-1}} (-1)^{\lambda_n} q^{\frac12\left(\lambda_n+a\right)^2+\left(\lambda_1+1\right)\left(r_1+\frac12\right)+\sum_{r=2}^n 
\left(\lambda_k-\lambda_{k-1}+1\right)\left(r_k+\frac12\right)}.
\]
Then the above shows that
\begin{equation*}\label{Frec}
F_{a,n}(r_1, \dots,r_n) =  \frac{q^{r_1 +\frac12}}{1-q^{r_1-r_2}}\Big(F_{a,n-1}\left(r_2,\dots,r_n\right)- F_{a,n-1}(r_1,r_3,\dots,r_n)\Big).
\end{equation*}
This gives inductively that 
$$F_{n,n}(r_1, \dots,r_n) =  \frac{q^{\frac{n}{2}}}{\prod_{1 \leq j < k \leq n}\left(1-q^{r_j-r_k}\right)}\sum_{1\leq j \leq n}P_{j}(q)G(r_j;\tau)
$$
with
$$G(a;\tau):=\sum_{\ell \geq 0}(-1)^{\ell} q^{\frac12(\ell + n)^2 + \ell\left(a +\frac12\right)}$$
and $P_j$ is a polynomial. Now 
\begin{align} \label{G-function}
 G(a;\tau) &=  (-1)^{n+a}q^{-n\left(a+\frac12\right) -\frac12a\left(a+1\right)}\sum_{ \ell \geq n +a}(-1)^{\ell}q^{\frac12\ell\left(\ell+1\right)},
\end{align}
which gives the claim.
\end{proof}

\begin{remark} The Laurent polynomials $P$ and $Q$ are uniquely determined.
\end{remark}

\subsection{Negative powers: no repetition}
We again assume that
$r_1 > r_2 > \cdots > r_n \geq 0$.
In this case, we obtain the following result.
\begin{theorem} We have 
\begin{equation} \label{negative}
\begin{split}
 {\rm Coeff}_{t^{-r_1-\frac12} \cdots t^{-r_n-\frac12}} F(t_1,...,t_n)={\rm CT}_{\zeta} \left( \prod_{j=1}^n \frac{1}{\left(1+\zeta  q^{r_j+\frac{1}{2}}\right)} \sum_{\ell \in \mathbb{Z}} \zeta^\ell q^{\frac{\ell^2}{2}}  \right).
\end{split}
\end{equation}
\end{theorem}
\begin{proof} Since the proof is  analogous to the proof of Theorem \ref{2.1}, we omit some details here. Recall
$$ {\rm Coeff}_{t^{-r_1-\frac12} \cdots t^{-r_n-\frac12}} F(t_1,...,t_n)=\eta(\tau) \times {\rm tr}_{\mathcal{F}_0} \prod_{j=1}^n \psi^*\left(r_j+\frac12\right)\psi\left(-r_j-\frac12\right) q^{L(0)-\frac1{24}}.$$
We apply the relation $\psi^*\left(r_j+\frac12\right)\psi\left(-r_j-\frac12\right)=1-\psi\left(-r_j-\frac12\right)\psi^*\left(r_j+\frac12\right)$.
This relation when inserted in the trace, together with the observation that the roles of  $\psi$ and $\psi^*$ are now switched, implies that the 
left hand-side of (\ref{negative}) equals
$${\rm CT}_{\zeta} \left( \prod_{j=1}^n \left(1-\frac{\zeta q^{r_j+\frac{1}{2}}}{\left(1+\zeta q^{r_j+\frac{1}{2}}\right)} \right) \sum_{\ell \in \mathbb{Z}} \zeta^\ell q^{\frac{\ell^2}{2}}  \right).$$
This gives the wanted relation.
\end{proof}
We do not compute the explicit Fourier expansion of (\ref{negative}) here  - this is done in Section 2.5 in a 
more general setup but instead we give an example.

\begin{example} ($n=1$) We have for $r\in\N_0$ 
$${\rm Coeff}_{t^{-r-\frac12}}  F(t)={\rm CT}_{\zeta} \left( \frac{1}{1+\zeta q^{r+\frac{1}{2}}} \sum_{\ell\in \mathbb{Z}} \zeta^\ell q^{\frac{\ell^2}{2}}  \right)
=\sum_{\ell\geq 0} (-1)^\ell q^{\frac{\ell^2}{2}+\ell\left(r+\frac{1}{2}\right)}.$$ 
This agrees with formula (3.51) in \cite{M}, which was obtained by using a completely different method.
\end{example}

\subsection{Negative and positive powers: colliding case}
We now discuss the case when in the $t$-expansion of $F(t_1,...,t_n)$ there are terms with positive and negative  powers of $t$.

\begin{proposition}
If $r_1,....,r_n \in \mathbb{Z}$ are such that for some $j$ and $k$ we have $-r_j-1/2=r_k+1/2,$
then 
$${\rm Coeff}_{t^{r_1+\frac12} \cdots t^{r_n+\frac12}} F(t_1,...,t_n)=0.$$ 
\end{proposition}
\begin{proof}  Because of 
\begin{align*} 
& \psi^*\left(r_j+\frac12\right)\psi\left(-r_j-\frac12\right)  \psi\left(r_k+\frac12\right)\psi^*\left(-r_k-\frac12\right) \\
& = \psi^*\left(r_j+\frac12\right) \underbrace{\psi\left(-r_j-\frac12\right)  \psi\left(-r_j-\frac12\right)}_{=0}\psi^*\left(r_j+\frac12\right)=0,
\end{align*}
the operator  ${\rm Coeff}_{t^{r_1+1/2} \cdots t^{r_n+1/2}} F(t_1,...,t_n)$ computes the trace 
of the zero operator. This directly yields the claim.
\end{proof}

\subsection{Repetitions: compression phenomenon}

Let $r_j \in \mathbb{Z}$ ($j=1,...,n$) possibly with repetition. Moreover, let $I_{(1)}$,...,$I_{(k)}$ be the partition of the set $\{1,...,n \}$:  
$$\{1,...,n \}=\bigsqcup_{j=1}^k I_{(j)},$$
 such that for all $j \in I_{(k)}$, the values $r_j$ are all equal to some $r_{(k)}$. 
 Then we get
\begin{lemma} \label{comp} We have
\begin{multline*}{\rm Coeff}_{t_1^{r_1+\frac12} \cdots t^{r_n+\frac12}_n} F(t_1,...,t_n)\\={\rm Coeff}_{\left(\prod_{j \in I_{(1)}} t_j\right)^{r_{(1)}+1} \cdots \left(\prod_{j \in I_{(k)}} t_j\right)^{r_{(k)}+1}} F\left(\prod_{j \in I_{(1)}} t_j,...,\prod_{j \in I_{(k)}} t_j\right).
\end{multline*}
\end{lemma}
\begin{proof} The claim is immediate, once we observe the formula 
$$\left( \psi^*\left(r_j+\frac12\right)\psi\left(-r_j-\frac12\right) \right)^k= \psi^*\left(r_j+\frac12\right)\psi\left(-r_j-\frac12\right)$$
and use the description of $F(t_1,...,t_n)$ as trace operator.
\end{proof}

We next consider an example.
\begin{example} Take $n=2$ and let $r=r_1=r_2>0$. Then, by Example \ref{ex1}, 
$${\rm Coeff}_{t_1^{r+\frac12} t_2^{r+\frac12}} F(t_1,t_2)={\rm Coeff}_{(t_1 t_2)^{r+\frac12}} F(t_1t_2)=\sum_{\ell\geq 0} (-1)^\ell q^{\frac{(\ell+1)^2}{2}+(\ell+1)\left(r+\frac{1}{2}\right)},$$
which also follows from Proposition 3.5 in \cite{M}.
\end{example}

\subsection{General formula: no repetitions} 
In view of the previous discussion and compression phenomena (cf. Lemma \ref{comp}),  without loss of generality, we may assume that 
$ r_1 > r_2 > \cdots > r_m \geq 0$ and  $ s_{m+1} >  \cdots > s_n \geq 0$ and that  $-r_j-1/2 \neq s_k+1/2$ for all $1 \leq j \leq m$ and $m+1 \leq k \leq n$.
We summarize everything into a single result.
\begin{theorem}\label{2.7}
Then, with $\text{\bf r} :=(r_1, \dots, r_m) \text{ and } \text{\bf s} :=(s_{m+1}, \dots, s_n)$,
\begin{equation*}
\begin{split}
G_{\text{\bf r}, \text{\bf s} }(q) :&=  {\rm Coeff}_{t^{r_1+\frac12}_1 \cdots t^{r_m+\frac12}_m t_{m+1}^{-s_{m+1}-\frac12} \cdots t_n^{-s_n-\frac12}} F(t_1,...,t_n) \\
&={\rm CT}_{\zeta} \left( \prod_{j=1}^m \frac{\zeta^{-1} q^{r_j+\frac{1}{2}} }{\left(1+\zeta^{-1} q^{r_j+\frac{1}{2}}\right)}  \prod_{j={m+1}}^n \frac{1}{\left(1+\zeta q^{s_j+\frac{1}{2}}\right)} \sum_{\ell \in \mathbb{Z}} \zeta^\ell q^{\frac{\ell^2}{2}} \right).
\end{split}
\end{equation*}
\end{theorem}

From this we obtain a $q$-series representation. 
\begin{proposition} With the assumptions as above, we have
\begin{align*}
G_{\text{\bf r}, \text{\bf s} }(q)  =\sum_{\substack{a_j\geq 0 \\ 1\leq j \leq n}}  (-1)^{\sum_{j=1}^n a_j} q^{\frac12\left(
\sum_{k=1}^m a_k + m - \sum_{k=m+1}^n a_k\right)^2+
\sum_{j=1}^m \left(a_j+1\right)\left(r_j+\frac{1}{2}\right) +  \sum_{j=m+1}^n a_{j}\left(s_{j}+\frac12\right)}.
\end{align*}

\end{proposition}

We next consider the case $n=2$ as an explicit example.
\begin{example} \label{11} ($n=2$) We have
\begin{align*}
G_{r,s}(\tau)&=\sum_{k, \ell \geq 0} (-1)^{k+\ell} q^{\frac12\left(k-\ell+1\right)^2+(k+1)\left(r+\frac{1}{2}\right)+\ell\left(s+\frac{1}{2}\right)} \\
& =\frac{1}{1-q^{r+s+1}} \left(\sum_{\ell\geq 1} (-1)^{\ell+1}q^{\frac{\ell^2}{2}+\ell\left(r+\frac{1}{2}\right)}-
q^{r+s+1} \sum_{\ell\geq 0} (-1)^\ell q^{\frac{\ell^2}{2}+\ell\left(s+\frac{1}{2}\right)} \right).
\end{align*}
\end{example}
We have the following generalization of Theorem \ref{2.2}.
\begin{theorem} \label{main-formula}
With $r_j$($1 \leq j \leq m$) and $s_k$($m+1 \leq k \leq n$) as above, the series
$$\prod_{\substack{1 \leq j \leq m \\ m+1 \leq k \leq n}} \left(1-q^{r_j+s_k+1}\right)  \prod_{1 \leq k < \ell \leq m} \left(1-q^{r_k -r_\ell}\right)  \prod_{m+1 \leq k < \ell \leq n} \left(1-q^{s_k -s_\ell}\right) G_{\text{\bf r}, \text{\bf s} }(q) $$
is a sum of series of the form $\sum_{\ell=0}^\infty (-1)^\ell q^{\frac{\ell^2}{2}+b\ell+c}$, where $b,c \in \frac{1}{2} \mathbb{Z}$. More precisely, it has the shape $P(q)\sum_{\ell \geq 0}(-1)^{\ell}q^{\frac{\ell(\ell+1)}{2}}+ Q(q)$ with $P$ and $Q$ being Laurent polynomials.
\end{theorem}
\begin{proof}

As in in the proof of Theorem \ref{2.2}, we set $\lambda_j := a_1 + \dots + a_j (1\leq j \leq m)$ and $\mu_{m+j} := a_{m+1} + \dots + a_{m+j}$ ($1\leq j \leq n-m$). Then $a_j = \lambda_j - \lambda_{j-1}$ for $2\leq j \leq m$ and $b_{m+j} = \lambda_{m+j} - \lambda_{m+j-1}$ for $2\leq j\leq n-m$. Thus
$$
G_{\text{\bf r}, \text{\bf s} }(q)   = F_{m,m, n-m} \left( r_1, \dots, r_m, s_{m+1}, \dots, s_n \right)
$$
with 
\begin{align*}
F_{a,j,k} 
&\left( r_1, \dots, r_j, s_{1}, \dots, s_k \right) := 
\sum_{\lambda_j, \mu_k\geq 0}(-1)^{\mu_j +\lambda_k} q^{\frac12 \left( \mu_j - \lambda_{k} + a \right)^2}
\\
&\times \sum_{\substack{0\leq \lambda_{\ell-1} \leq \lambda_\ell \\ 0\leq \mu_{\nu-1} \leq \mu_\nu \\2\leq \ell \leq j, 2\leq \nu\leq k}} 
q^{\left( \lambda_1 +1\right) \left( r_1 +\frac12\right) +
\sum_{\ell=2}^j 
 \left( \lambda_\ell -\lambda_{\ell-1}+1\right) \left( r_\ell+\frac12\right)
+\mu_1\left(s_1+\frac12\right)+\sum_{\ell=2}^k\left(\mu_\ell-\mu_{\ell-1}\right)\left(s_\ell+\frac12\right)}.
\end{align*}
Carrying out the summation over $\lambda_1$ gives
\begin{multline*}
F_{a,j, k} \left( r_1, \dots, r_j, s_{1}, \dots, s_k \right)= 
\frac{q^{r_1 +\frac12}}{1-q^{r_1 - r_2}}
\bigg( F_{a,j-1, k} \left( r_2, \dots, r_j, s_{1}, \dots, s_k \right) .\\ - F_{a, j-1, k} \left( r_1,r_3,  \dots, r_j, s_{1}, \dots, s_k \right) \bigg).
\end{multline*}
Similarly summing over $\mu_1$ yields
\begin{multline*}
F_{a,j, k} \left( r_1, \dots, r_j, s_{1}, \dots, s_k \right)= 
\frac{1}{1-q^{s_1 - s_2}}
\bigg( F_{a,j, k-1} \left( r_1, \dots, r_j, s_{2}, \dots, s_k \right) .\\ - q^{s_1-s_2} F_{a, j, k-1} \left( r_1, \dots, r_j, s_{1}, s_3,  \dots, s_k \right) \bigg).
\end{multline*}
We thus obtain inductively that 
\begin{multline*}
F_{m,m,n-m} \left( r_1, \dots , r_m, s_{m+1}, \dots, s_n \right) \\= 
\frac{q^{\frac{m-1}{2}}}{\prod_{1\leq j < k \leq m} \left( 1- q^{r_j - r_k} \right) \prod_{m+1 \leq j < k \leq n} \left( 1- q^{s_j - s_k} \right)}
 \sum_{\substack{1\leq j \leq m \\m+1\leq k \leq n}} P_{j,k} (q) G \left( r_j, s_k; \tau \right)
\end{multline*}
with
$$
G \left(r,s; \tau \right) :=
\sum_{k, \ell \geq 0} (-1)^{k+\ell} q^{\frac12 \left( k - \ell +m\right)^2 + k 
\left( r +\frac12 \right) + \ell \left( s+\frac12\right)}.
$$
To see that $G(r,s;\tau)$ has the desired shape, in view of Example \ref{11}, we may assume that $m \geq 2$. We rewrite
\begin{align*}
G \left(r,s; \tau \right)&=q^{-\left(r+\frac12\right)(m-1)} (-1)^{m-1} \sum_{k \geq m-1 \atop{ \ell \geq 0}} (-1)^{k+\ell} q^{\frac12 \left( k - \ell +1\right)^2 + k 
\left( r +\frac12 \right) + \ell \left( s+\frac12\right)} \\
&=q^{-\left(r+\frac12\right)(m-1)} (-1)^{m-1} \biggl( \sum_{\ell \geq 0} (-1)^{k+\ell} q^{\frac12 \left( k - \ell +1\right)^2 + k 
\left( r +\frac12 \right) + \ell \left( s+\frac12\right)} \\ 
& \qquad \qquad -\sum_{0\leq k \leq m-2}(-1)^k q^{k \left(r+\frac12 \right)} \sum_{\ell \geq 0} (-1)^\ell q^{\frac12 \left( k - \ell +1\right)^2 + \ell \left( s+\frac12\right)}\biggr).
\end{align*}
For the first double sum in the parentheses we use  Example \ref{11} and for the second one we slightly rewrite the $\ell$-summation and then apply identity (\ref{G-function}).
 \end{proof}

\section{Asymptotics}
In this section we study asymptotic properties of $F_{\text{\bf r}}$ and of  $G_{\text{\bf r},\text{\bf s}}$
towards the cusp  $0$. We should mention that in \cite{BO},  asymptotic properties of 
$F(e^{2 \pi i u_1},...,e^{2 \pi i u_n})$ were studied in connection to quasi-modular forms.

\subsection{Asymptotics  of $F_{\text{\bf r} }$}

\begin{theorem} \label{F-asym}
We have, as $ y \rightarrow 0$, 
\begin{multline*}
F_{\text{\bf r}}\left( e^{-2\pi y}\right) = \frac1{2^n}\sum\limits_{{{{\nu_1, \dots, \nu_n\geq 0\atop{0\leq \ell_k\leq \nu_k}}}\atop{1\leq k\leq n}}\atop{\ell\text{ even}}}(-1)^{\nu}2^\nu \pi^{\nu-\frac{\ell}{2}}i^\ell  \frac{\Gamma \left(\frac{\ell+1}{2} \right)}{\sqrt{\pi}} \prod_{j=1}^n  \frac{E_{\nu_j}(1)}{\ell_j!(\nu_j-\ell_j)!}\left(r_j+\frac12\right)^{\nu_j-\ell_j}y^{\nu-\frac{\ell}{2}}\\
+O\left(e^{-ay}\right),
\end{multline*}
where $a\in\R^+, \ell:=\sum_{k=1}^n \ell_k$, $\nu:=\sum_{k=1}^n\nu_k$, and $E_k(x)$ denotes the $k$th Euler polynomial.
 \end{theorem}
\begin{proof}
We abbreviate ($\zeta:=e^{2\pi iz}$ throughout)
\begin{align*}
\mathcal{F} (\tau)&:= F_{\text{\bf r}} \left( e^{2\pi i \tau}\right),\\
\mathcal{F} \left( z; \tau \right) &:=
\prod_{j=1}^n \frac{ \zeta^{-1}q^{r_j+\frac12}}{\left( 1+\zeta^{-1}q^{r_j+\frac12}\right)} \sum_{\ell\in\Z} \zeta^\ell q^{\frac{\ell^2}2}.
\end{align*}
Then, by Theorem \ref{2.1}, $\mathcal{F} (\tau) = \mathrm{CT}_\zeta \mathcal{F} \left( z; \tau \right)$ and Cauchy's Theorem yields
\[
\mathcal{F}(\tau)=\int_{-\frac12}^{\frac12}\mathcal{F}(z; \tau)dz=\int_{-\frac12}^{\frac12}\mathcal{F}(-z; \tau)dz.
\]
Let
\[
\vartheta(z; \tau):=\sum_{\ell\in\Z} \zeta^\ell q^{\frac{\ell^2}{2}}.
\]
Using the usual theta transformation law, we get
(\cite{SS}, page 290)
\[
\vartheta(z; \tau)=(-i\tau)^{-\frac12}\sum_{\ell\in\Z} e^{-\frac{\pi i}{\tau}(\ell-z)^2}.
\]
Since $-\frac12<z<\frac12$, we thus obtain that, up to an exponentially smaller term,
\begin{equation}\label{thetaas}
\vartheta(z; iy)\sim\frac{e^{-\frac{\pi z^2}{y}}}{\sqrt{y}}.
\end{equation}
Moreover
\begin{equation}\label{Eulergen}
\begin{split}
\frac{\zeta q^{r+\frac12}}{1+\zeta q^{r+\frac12}}
&=\frac12\sum_{\nu\geq 0}\frac{E_{\nu}(1)}{\nu!}\left(2\pi iz+2\pi i\left(r+\frac12\right)\tau\right)^{\nu}\\
&=\frac12 \sum\limits_{\nu\geq 0\atop{0\leq \ell\leq \nu}}\frac{E_{\nu}(1)}{\nu!}\binom{\nu}{\ell}(2\pi i)^{\nu}\left(r+\frac12\right)^{\nu-\ell}z^{\ell}\tau^{\nu-\ell},
\end{split}
\end{equation}
where $\ell$ and $\nu$ should not be confused with those used in the statement of the theorem.
Thus 
\begin{multline*}
\mathcal{F}(iy)
\sim\frac1{2^n\sqrt{y}}\sum\limits_{{\nu_1, \dots, \nu_n\geq 0\atop{0\leq \ell_k\leq \nu_k}}\atop{1\leq k\leq n}}\prod_{j=1}^n\left[\frac{E_{\nu_j}(1)}{\ell_j!(\nu_j-\ell_j)!}
(-1)^{\nu_j}\left(r_j+\frac12\right)^{\nu_j-\ell_j} y^{\nu_j-\ell_j}(2\pi)^{\nu_j} (-i)^{\ell_j}\right]\\\times \int_{-\frac12}^{\frac12}z^{\ell}e^{-\frac{\pi z^2}{y}}dz\\
=\frac1{2^n\sqrt{y}}\sum\limits_{{{\nu_1, \dots, \nu_n\geq 0\atop{0\leq \ell_k\leq \nu_k}}\atop{1\leq k\leq n}}\atop{\ell\text{ even}}}(-1)^{\nu} y^{\nu-\ell}(2\pi)^\nu i^\ell\prod_{j=1}^n
\frac{E_{\nu_j}(1)}{\ell_j!(\nu_j-\ell_j)!}\left(r_j+\frac12\right)^{\nu_j-\ell_j} \mathcal{I}(\ell; y),
\end{multline*}
where for $\ell$ even
\[
\mathcal{I}(\ell; y):=\int_{-\frac12}^{\frac12} z^{\ell} e^{-\frac{\pi z^2}{y}}dz.
\]
Note that, with an exponentially smaller error term,
\[
\mathcal{I}(\ell; y)\sim\left(\frac{y}{\pi}\right)^{\frac{\ell+1}{2}}\Gamma\left(\frac{\ell+1}{2}\right).
\]
This follows by turning the integral in an integral over $\R$ and using that for $a\in \R^{+}$
$$
\int_{\R} z^{\ell} e^{-a z^2}dz = a^{\frac{\ell+1}{2}}\Gamma\left(\frac{\ell+1}{2}\right).
$$
The error is given by bounding the lower incomplete gamma function.
Combining the above yields the required asymptotic expansion.
\end{proof}
From Theorem \ref{F-asym}, we directly obtain the following first terms in the asymptotic expansion.
\begin{corollary}
We have with $r:=\sum_{k=1}^n r_k$
$$
F_{\text{\bf r}} \left( e^{-2\pi y}\right) = \frac1{2^n}\left(1-\pi \left(r+ \frac{n(n+1)}{4}
\right)y +O\left(y^2\right)\right).
$$
\end{corollary}
\begin{proof}
The claim follows directly from Theorem \ref{F-asym}, noting that $E_0(1)=1, E_1(1)=\frac12$, and $E_2(1)=0$.
\end{proof}


\begin{remark} {\rm In view of the previous theorem, we consider the
asymptotic behavior of the Fourier coefficients of $F_{\text{\bf r}}$
$$F_{\text{\bf r}}(q)=: \sum_{\ell \geq 0} a_\ell q^{\ell}.$$
But a closer inspection of the coefficients $a_m$ shows a quite erratic behavior. For example, if we take ${\bf r}=(2,3,4,5)$, then in the expansion of $F_{\bf r}$ we
have
$$a_{43}=2, \ \ a_{100}=-7, \ \ a_{153}=18, \ \ a_{245}=-2, \ \ a_{538}=-81, \ \ a_{713}=112, \ \ a_{894}=-4.$$
On the other hand 
$$\frac{F_{\text{\bf r}}(q)}{(q;q)_\infty}=: \sum_{\ell \geq 0} b_\ell q^{\ell}$$
has non-negative coefficients (see Remark \ref{rem-1}) that are much better behaved. For example, for ${\bf r}$ as above,  the coefficients $b_\ell$ seem to  
be positive and increasing for all $\ell\geq 25$. We have the following asymptotic behavior, with $p(\ell)$ denoting the number of partitions of $\ell$. }
\end{remark}
\begin{theorem} \label{asym-coeff}
We have, as $\ell\rightarrow\infty$,
$$b_{\ell} \sim \frac{p(\ell)}{2^n}.$$
\end{theorem}
\begin{proof}
The claim follows by using Wright's version of the Circle Method \cite{Wr}. The key is that one can prove, similarly as in Theorem \ref{F-asym}, that
\[\mathcal{F}(\tau)\sim\frac{1}{2^n }\]
as $\tau\rightarrow 0$ with $\tau=u+iv$ and $|u|\leq v$. This gives the appropriate bound on the major arc in Wright's Circle Method. On the minor arcs one uniformly approximates the partial theta function similarly to Wright (see the calculations leading to formula (4.6) of \cite{Wr}). 
We leave the details to the interested reader.
\end{proof}

\subsection{Asymptotic properties of $G_{\text{\bf r},\text{\bf s}}$}
Similarly to the case of $F_{\text{\bf r}}$, we obtain for $G_{\text{\bf r, s}}$
\begin{theorem} \label{main}
We have as, $y \to 0$,
\begin{align*}
G_{\text{\bf r},\text{\bf s}} \left( e^{-2\pi y}\right)
&=\frac1{2^n}\sum\limits_{{{\nu_1, \dots, \nu_n\geq 0\atop{0\leq \ell_k\leq \nu_k}}\atop{1\leq k\leq n}}\atop{{\ell\text{ even}}}}(-1)^\nu2^{\nu} \pi^{\nu-\frac{\ell}{2}}i^\ell \frac{ \Gamma\left(\frac{\ell+1}{2}\right)}{\sqrt{\pi}} \prod_{j=1}^m \frac{E_{\nu_j}(1)(-1)^{\ell_j}}{\ell_j!\left(\nu_j-\ell_j\right)!}\left(r_j+\frac12\right)^{\nu_j-\ell_j}\\
&\qquad \times \prod_{j=m+1}^n
\frac{ E_{\nu_j}(0)}{\ell_j!\left(\nu_j-\ell_j\right)!}\left(s_j+\frac12\right)^{\nu_j-\ell_j} 
 y^{\nu -\frac{\ell}{2}}+O\left(e^{-ay}\right),
\end{align*}
where $a\in\R^+$,
$\ell:=\sum_{k=1}^n \ell_k$, and $\nu:=\sum_{k=1}^n \nu_k$, 
\end{theorem}
\begin{proof}
We proceed as in the proof of Theorem \ref{F-asym} and set
\begin{align*}
\mathcal{G} (\tau) &:= G_{\text{\bf r},\text{\bf s}} \left( e^{2\pi i \tau}\right),\\
\mathcal{G} \left( z; \tau \right) &:=
\prod_{j=1}^m \frac{\zeta^{-1} q^{r_j +\frac12}}{\left( 1 + \zeta^{-1} q^{r_j +\frac12}\right)}
\prod_{j=m+1}^{n} \frac{1}{\left( 1+\zeta q^{s_j + \frac12}\right)}\vartheta(z; \tau).
\end{align*}
Then, by Theorem \ref{2.7}, $\mathcal{G} (\tau) = \text{CT}_\zeta \mathcal{G}(z; \tau)$ and Cauchy's Theorem yields
\[
\mathcal{G}(\tau)=\int_{-\frac12}^{\frac12} \mathcal{G}(-z; \tau)dz.
\]
We now use \eqref{thetaas}, \eqref{Eulergen}, and
\begin{align*}
\frac1{1+\zeta^{-1} q^{s+\frac12}}
&=\frac12 \sum_{\nu\geq 0} \frac{E_{\nu}(0)}{\nu!}\left(-2\pi iz+2\pi i\left(s+\frac12\right)\tau\right)^{\nu}\\
&=\frac12\sum\limits_{\nu\geq 0\atop{0\leq \ell\leq \nu}}\frac{E_{\nu}(0)}{\nu!}\binom{\nu}{\ell}(-1)^{\nu}(2\pi i)^{\nu}\left(s+\frac12\right)^{\nu-\ell}\tau^{\nu-\ell}.
\end{align*}
Thus
\begin{align*}
\mathcal{G}(iy)
&\sim\frac1{2^n\sqrt{y}}\sum\limits_{{{\nu_1, \ldots, \nu_n\geq 0\atop{0\leq \ell_k\leq \nu_k}}\atop{1\leq k\leq n}}}(-1)^\nu y^{\nu-\ell}(2\pi)^\nu i^\ell\prod_{j=1}^m \frac{E_{\nu_j}(1)(-1)^{\ell_j}}{\ell_j!\left(\nu_j-\ell_j\right)!}\left(r_j+\frac12\right)^{\nu_j-\ell_j}\\
&\qquad\times\prod_{j=m+1}^n
\frac{ E_{\nu_j}(0)}{\ell_j!\left(\nu_j-\ell_j\right)!}\left(r_j+\frac12\right)^{\nu_j-\ell_j}\mathcal{I}\left(\ell; y\right),
\end{align*}
which directly yields the claim. 
\end{proof}
From Theorem \ref{main}, we may again determine the leading asymptotic terms.
\begin{corollary}
We have with $r:=\sum_{k=1}^m r_k$ and $s:=\sum_{k=m+1}^n s_k$
\[
G_{\text{\bf r},\text{\bf s}}\left(e^{-2\pi y}\right)=\frac1{2^n}\left(1-\pi\left(r-s+m+\frac{n(n-3)}4 \right)y+ O\left(y^2\right)\right).
\]
\end{corollary}

As in Theorem \ref{asym-coeff}, from Theorem \ref{main} and \cite{Wr}, we easily infer, writing $G_{\text{\bf r},\text{\bf s}}(q)=:\sum_{\ell=0}^\infty c_\ell q^{{\ell}},
$
\begin{corollary} \label{asym-coeff-2}
We have
$$c_{2\ell} \sim \frac{p(\ell)}{2^n}.$$
\end{corollary}

\section{Higher level Bloch-Okounkov $n$-point functions}

There are many possible extensions and generalizations of the concept of Bloch-Okounkov $n$-point function \cite{M,EO,CW,W}. The most obvious one comes
from the consideration of a single uncharged free fermion vertex superalgebra \cite{M,W}. Although a recursion can be derived in this case, 
an explicit formula as in Theorem 1.1 is not known. 
Another level of generalization comes from considerations of tensor products of vertex (super) algebras.
 ``Higher level'' infinite wedge space is one such example - it is simply the tensor product of $\ell \in \mathbb{N}$
copies of the (rank one) infinite-wedge space discussed earlier. 

Let us denote rank-$\ell$ fermonic generators with $\psi_j(r+1/2)$,
$\psi^*_j(r+1/2)$ ($j=1,...,\ell$), where the anti-bracket relations are 
$$\left[\psi_j\left(r+\frac12\right),\psi^*_k\left(p+\frac12\right)\right]=\delta_{j,k} \delta_{r+p+1,0}.$$
 The relevant spaces are now $V=\mathcal{F}^{\otimes^\ell}$ and its charge zero subspace is 
$\mathcal{F}_0^{\otimes^\ell}$. Here the charge of the tensor product is defined as the sum of charges of tensor components.
Again, other non-zero charge subspaces can be easily handled by using shifting.
We set
\[
F^{(\ell)}(t_1,...,t_n):= \eta(\tau)^\ell \times {\rm tr}_{V} A^{(\ell)}(t_1) \cdots A^{(\ell)}(t_n) q^{L(0)-\frac{\ell}{24}},
\]
where 
$$A^{(\ell)}(t):=\prod_{j=1}^\ell \sum_{r\in \mathbb{Z}} \psi^*_j\left(r+\frac12\right) \psi_j\left(-r-\frac12\right) t^{-r-\frac12},$$
(here ${}^{(\ell)}$ not to be confused with $\ell$-th derivative). 

\begin{proposition}
We have
$$F^{(\ell)}(t_1,...,t_n)=  F\left(t_1,...,t_n\right)^{\ell}.$$
In particular, $F^{(\ell)}(t)$ has a meromorphic continuation such that 
$$F^{(\ell)}(t)=\frac{1}{\Theta(t)^{\ell}}.$$
\end{proposition}
\begin{proof}
By using the multiplicative property of the trace and the relation $L(0)=\sum_{j=1}^n L^{(j)}(0)$, we get  
\begin{align*}
&\eta(\tau)^\ell \cdot {\rm tr}_{V} A^{(\ell)} (t_1) \cdots A^{(\ell)} (t_n) q^{L(0)-\frac{\ell}{24}} \\
&=\eta(\tau)^\ell \cdot {\rm tr}_V  \prod_{j=1}^\ell \sum_{r\in \mathbb{Z}} \psi^*_j\left(r+\frac12\right) \psi_j \left(-r-\frac12\right) t_1^{-r-\frac12} \cdots \\
& \qquad \cdots \prod_{j=1}^\ell \sum_{r \in \mathbb{Z}} \psi^*_j \left(r+\frac12\right) \psi_j \left(-r-\frac12\right) t_n^{-r-\frac12}q^{L(0)-\frac{\ell}{24}}  \\
& =\eta(\tau)^\ell \prod_{j=1}^{\ell} {\rm tr}_{\mathcal{F}_0} A^{(1)}(t_1) \cdots A^{(1)}(t_n) q^{L^{(j)}(0)-\frac1{24}}={F(t_1,...,t_n)^{\ell}},
\end{align*}
as required.
\end{proof}

\begin{remark} {\em Studying properties of the Fourier coefficients of $F^{(\ell)}(t_1,...,t_n)$, for $\ell \geq 2$, is an interesting problem to which we hope 
to return in a  future publication. The only case that is currently well-understood is $F^{(\ell)}(t)$, due to \cite{BCR}.}
\end{remark}

\end{document}